\documentclass{amsart}

\usepackage[utf8]{inputenc}
\usepackage{lmodern}
\usepackage[T1]{fontenc}
\usepackage[textwidth=6.2in, textheight=7.8in, hcentering]{geometry}
\usepackage[pagebackref=true,breaklinks=true,colorlinks]{hyperref}
\usepackage{amssymb}
\usepackage{color}
\usepackage[all,cmtip]{xy}
\usepackage{amsmath}
\usepackage{amsthm}
\usepackage{amscd}
\usepackage{amsfonts, enumitem}
\usepackage{tikz}
\usetikzlibrary{matrix,arrows,decorations.pathmorphing}

\newtheorem{dummy}{anything}[section]
\newtheorem{theorem}[dummy]{Theorem}

\newtheorem{lemma}[dummy]{Lemma}
\newtheorem{proposition}[dummy]{Proposition}
\newtheorem{corollary}[dummy]{Corollary}

\theoremstyle{definition}

\newtheorem{notation}[dummy]{Notation}
\newtheorem{remark}[dummy]{Remark}

\newtheorem{condition}[dummy]{Condition}

\newcommand{\bbS}{\mathbb S}
\newcommand{\bbZ}{\mathbb Z}

\DeclareMathOperator{\Aut}{Aut}
\DeclareMathOperator{\Image}{im}
\DeclareMathOperator{\Out}{Out}
\DeclareMathOperator{\Inn}{Inn}
\DeclareMathOperator{\md}{mod}
\newcommand{\cP}{\mathcal P}

\newcommand{\sM}{\mathcal{M}}
\newcommand{\KO}{{{K}}}

\begin{document}

		\title{On smooth manifolds with the homotopy type of a homology sphere}
		\author{Mehmet Akif Erdal}
	\address{Department of Mathematics,
		Bilkent University, 06800, Ankara, Turkey}
	\email{merdal@fen.bilkent.edu.tr}
		
		\begin{abstract}
In this paper we study $\sM(X)$, the set of diffeomorphism classes of smooth manifolds with the simple homotopy type of $X$, via a map $\Psi$ from $\sM(X)$ into the quotient of $\KO(X) = [X, BSO]$ by the action of the group of homotopy classes of simple self equivalences of $X$. The map $\Psi$ describes which bundles over $X$ can occur as normal bundles of manifolds in $\sM(X)$. We determine the image of $\Psi$ when $X$ belongs to a certain class of homology spheres. In particular, we find conditions on elements of $\KO(X)$ that guarantee they are pullbacks of normal bundles of manifolds in $\sM(X)$.
		\end{abstract}

 	\maketitle

\section{Introduction}\label{intro}
Unless otherwise stated, by a manifold we mean a smooth, oriented, closed manifold with dimension greater than or equal to $5$. Given a simple Poincar\'{e} complex $X$ with formal dimension $m$, a classical problem in topology is to understand the set of diffeomorphism classes of smooth manifolds in the simple homotopy type of $X$. For such an aim, a fundamental object to study is the smooth simple structure set $\mathcal{S}^s (X)$ (see \cite{surveywall} page 125-126 for notation and details). Elements of $\mathcal{S}^s (X)$ are equivalence classes of simple homotopy equivalences $\omega:M\to X$ from an $m$-dimensional manifold $M$. Two such homotopy equivalences $\omega_1:M_1\to X$ and $\omega_2:M_2\to X$ are said to be equivalent if there is a diffeomorphism $g:M_1\to M_2$ such that $\omega_1$  is  homotopic to the composition $\omega_2\circ g$. An element of  $\mathcal{S}^s (X)$ is called a simple smooth manifold structure on $X$. Note that composition of an element in  $\mathcal{S}^s (X)$ with a simple self equivalence of $X$ gives another element in $\mathcal{S}^s (X)$, although the manifold is still the same. Hence, we need to quotient out simple self equivalences of $X$ in order to get the set of diffeomorphism classes of smooth manifolds in the simple homotopy type of $X$. Denote  $\Aut_s(X)$ the group of homotopy classes of simple self equivalences of $X$. Then $\Aut_s(X)$ acts on $\mathcal{S}^s (X)$ by composition. The set of diffeomorphism classes of smooth manifolds in the simple homotopy type of $X$, $\sM(X)$, is defined as the set of orbits of $\mathcal{S}^s (X)$ under the action of $\Aut_s(X)$, i.e. $\sM(X):=\mathcal{S}^s (X)/\Aut_s(X)$.

Let $\KO(X)$ denote the group of homotopy classes of maps $[X,BSO]$ (here, we abandon the traditional notation $\widetilde{KSO}(X)$ for simplicity). Every simple homotopy equivalence $X\to X$ induces an automorphism on $\KO(X)$. Let $\overline{\Aut_s}(\KO(X))$ denote the subgroup of $\Aut(\KO(X))$ that consist of automorphisms induced by the simple self equivalences of $X$. There is a canonical action of  $\overline{\Aut_s}(\KO(X))$ on $\KO(X)$ again given by composition. We denote by $\mathfrak{K}(X) $ the set of orbits of $\KO(X)$ under the action of $\overline{\Aut_s}(\KO(X))$. 

As pointed out in \cite{surveywall}  computations of $\Aut_s(X)$ and $\mathcal{S}^s (X)$ are in general difficult, so does the computation of $\sM(X)$. On the other hand, computations  of $\KO(X)$ and $\overline{\Aut_s}(\KO(X))$ are easier in most cases as $\KO(-)$ is a (generalized) cohomology group (see \cite{ahss}). In this paper, we compare $\mathfrak{K}(X) $ with $\sM(X)$ where $X$ belongs to a certain class of Poincar\'{e} complexes.

There is a map $\Psi:\sM(X)\to \mathfrak{K}(X)$ defined by $[\omega:M\to X]\mapsto (\omega^{-1})^*(\nu)$ where $\nu:M\to BSO$ denotes the normal bundle of $M$ (see Proposition \ref{psi}). For a prime $q$, by a $\bbZ/q$-homology $m$-sphere we mean a simple Poincar\'{e} complex $X$ of formal dimension $m$ such that $H_*(X;\bbZ/q)\cong H_*(S^m;\bbZ/q)$. For a general reference to Poincar\'{e} complexes we refer to \cite{wallpc} and  \cite{spivak}. Our purpose is to determine the image of $\Psi$ for certain such homology spheres. Here we also assume that such a homology sphere admits a degree one normal map (equivalently the Spivak normal fibration has a vector bundle reduction), since otherwise the problem is trivial.

Let $m$ be an odd number and $S$ be a subset of the set of primes between $(m+4)/4$ and $(m+2)/2$. Denote by $\mathfrak{K}(X)_{(S,\mathtt{q}_1)} \subset  \mathfrak{K}(X)$ the set of orbits in $\mathfrak{K}(X)$ that can be represented by elements  $\xi$ such that for each $p$ in $S$ the first $ \md p $ Wu class of $\xi$ satisfy the identity $\mathtt{q}^p_1(\xi)+\mathtt{q}^p_1(X)=0$ (see \cite{Wu} or \cite{milsta}). Note that $\KO(S^m)\cong 0$ for $m$ odd and $m\neq 1 (\md 8)$. The main object of this paper is to prove  the following:
\begin{theorem}\label{main}Let $m$ be an odd number and $S$ be a subset of the set of primes between $(m+4)/4$ and $(m+2)/2$. Let  $X$ with a given map $f:S^m\to X$ be a $\bbZ/q$-homology $m$-sphere, so that $f$ is a $\bbZ/q$-homology isomorphism for every prime $q< (m+2)/2$ with $q\notin S$. Assume further that  $\pi_1(X)$ is of odd order. Then, the image of $\Psi$ consists of orbits in $ \mathfrak{K}(X)_{(S,\mathtt{q}_1)} $ that are represented by elements in the kernel of $f^*:\KO(X)\to \KO(S^m)$. In particular, if $m\neq 1(\md 8)$, then the image of $\Psi$ is $ \mathfrak{K}(X)_{(S,\mathtt{q}_1)} $. Furthermore, if $S=\emptyset$, then $\Psi$ is surjective.
\end{theorem}

Observe that as  $S$ gets larger, the image of $\Psi$ gets smaller. In particular, if $T=\{q \text{ prime }:   q< \frac{m+2}{2}  \}\backslash S$, then we do not need to make the assumption of Theorem \ref{main} on the $ \md q $ Wu classes for the primes   $q \in T$. On the other hand, for an odd prime $q$, $X$ being a  $\bbZ/q$-homology sphere implies there is no $q$-torsion in $K(X)$. Hence, primes in $T$ also affect the image of $\Psi$.

It is well-known, due to \cite{wall}, that a degree one normal map can be surgered to a simple homotopy equivalence if and only if the associated surgery obstruction vanishes. The main result in \cite{bak} states that if $\pi_1(X)$ is of odd order, then the odd dimensional surgery obstruction groups, $L^s_{m}(\bbZ[\pi_1(X)])$,  vanish, i.e. every degree one normal map can be surgered to a simple homotopy equivalence. An essential step in the proof of Theorem \ref{main}  is that, under the stated conditions, an element in $\KO(X)$ admits a degree one normal map if an only if it is in the kernel of $f^*$ and it has the same $\md p $ Wu classes as the Spivak normal fibration of $X$ for each $p \in S$. In particular,  if $m\neq 1(\md 8)$, then bundles admitting degree one normal maps are completely determined by their $\md p $  Wu classes for $p \in S$. In the case when $S=\emptyset$ and $m\neq 1(\md 8)$ every stable vector bundle over $X$ admits a degree one normal map. The rest is  to determine the action of $\overline{\Aut_s}(\KO(X))$ on $\KO(X)$, which is given by the restriction of the canonical action of $\Aut(\KO(X))$.

Some examples of such $X$ come from the smooth spherical space forms. Some applications of our theorem  are discussed in Section \ref{examples}. Smith theory may also provide examples, although we do not mention any such example in this note.

\section{Notation and Preliminaries}\label{prelim}
Let $X$ be a  simple Poincar\'{e}  complex with formal dimension $m$. Given a  stable vector bundle $\xi:X\to BSO$, a $\xi$-manifold is a manifold whose stable normal bundles lifts to $X$ through ${\xi}$, we refer \cite{stong} for more details (in \cite{stong} such objects are called $(B,f)$-manifolds). We denote by $\Omega_k(\xi)$ the cobordism group of $k$-dimensional $\xi$-manifolds. An element of $\Omega_k(\xi)$ is often denoted by $[\rho:M\to X]$, where $M$ is a $k$-dimensional manifold and $\rho$ is a lifting of its stable normal bundle to $X$ through $\xi$, and brackets denote the homotopy class of such liftings (see \cite{teichner} Proposition 2 for this notation). Such a map $\rho$ is called a \emph{normal map}, and if degree of $\rho$ is equal to $1$, i.e. $\rho_*[M]=[X]\in H_m(X;\bbZ)$, then it is called a \emph{degree one normal map} (see \cite{lueck} Definition 3.46.). Due to the Pontrjagin-Thom construction, the group  $\Omega_k(\xi)$ is isomorphic to $k$-th homotopy group of  $M\xi$, the Thom spectra associated to ${\xi}$ \cite{thom}.

Our primary tool is the James spectral sequence, which is a variant of the Atiyah-Hirzebruch spectral sequence (see \cite{teichner}, Section II). Let $h$ be a generalized homology theory represented by a connective spectrum,  $F\to X\stackrel{f}\to B$  be  an $h$-orientable fibration with fiber $F$ and ${\xi:X\to BSO}$ be  a stable vector bundle. The James spectral sequence for $h$, $f$ and $\xi$ has $E_2$-page $E^2_{s,t}=H_s(B;h_t(M\xi|_F))$ and converges to $h_{s+t}(M\xi)$. In the case when $h$ is the stable homotopy, the edge homomorphism of this spectral sequence coming from the base line is as follows:
\begin{proposition}[see \cite{teichner} Proposition 2]\label{edge}
	The edge homomorphism of the James spectral sequence for stable homotopy, $f:X\to B$ and ${\xi:X\to BSO}$ is a homomorphism $\mathfrak{ed}: \Omega_n(\xi) \to H_n(B,\bbZ)$ given by $$\mathfrak{ed}[\rho:M\to X]= f_*\circ\rho_*[M]$$ for every element $[\rho:M\to X]\in \Omega_n(\xi) $.
\end{proposition}
 The Atiyah-Hirzebruch spectral sequences for $M\xi$ is isomorphic to the James spectral sequence for stable homotopy, $id:X\to X$ and ${\xi:X\to BSO}$. This follows from the fact that $M\xi|_*$ is the sphere spectrum. This isomorphism is given by the Thom isomorphism (see proof of Proposition 1 in \cite{teichner}). In this paper, we will only use this edge homomorphism of  the James spectral sequence for the stable homotopy, the  identity map $id:X\to X$ and a given stable vector bundle ${\xi:X\to BSO}$. In this case $\mathfrak{ed}: \Omega_n(\xi) \to H_n(X,\bbZ)$ is the map given by $[\rho:M\to X]\mapsto \rho_*[M]$ for every element $[\rho:M\to X]\in \Omega_n(\xi) $. The other edge homomorphism for this spectral sequence will be denoted by $\mathfrak{\bar{ed}}:\pi_*(\bbS)\to \pi_*( M\xi) $, where $\bbS$ denotes the sphere spectrum.

Recall that $\mathfrak{K}(X) $ denotes the quotient $\KO(X)/\overline{\Aut_s}(\KO(X)) $ (see Section \ref{intro}). Let $\Phi:\KO(X)\to \mathfrak{K}(X)$ be the quotient map. We define a map $\Psi$ from $\sM(X)$ to $\mathfrak{K}(X)$ as follows: Let $M$ be a smooth manifold equipped with a simple homotopy equivalence $\omega:M\to X$ and let $\nu$ be the stable normal bundle of $M$. Let  $g:X\to M$ be the homotopy inverse of $\omega$. Then the pullback bundle $g^*(\nu)$ defines an element in $\KO(X)$. If $[M]$ is the diffeomorphism class of $M$ in $\sM(X)$, we define $\Psi[M]:=\Phi(g^*([\nu]))$.
\begin{proposition}\label{psi}
	$\Psi$ is well defined.
\end{proposition}
\begin{proof}
	Let $K$ be another manifold in the orbit $[M]$ with normal bundle $\kappa$, with a diffeomorphism $t:K\to M$ and with a simple homotopy equivalence $h:X\to K$. Since $t^*([\nu])=[\kappa]$, we have $h^*t^*([\nu])=h^*([\kappa])$. Since $\omega :M\to X$ is the homotopy inverse of $g$, we have $h^*([\kappa])=h^*t^*([\nu])=h^*t^*\omega^*g^*([\nu])$. Hence, $h^*([\kappa])$ and $g^*([\nu])$ differ by an automorphism  in $\overline{\Aut_s}(\KO(X))$ as the composition $ \omega\circ t\circ h$ is homotopic to a simple self homotopy equivalence of $X$. By definition, in $\mathfrak{K}(X)$ they are the same.	
\end{proof}

Let $p=2b+1$ be an odd prime. For a vector bundle, or in general a spherical fibration, $\xi$ over $X$, there exist cohomology classes $\mathtt{q}^p_k(\xi)$ in $ H^{4bk}(X;\bbZ/p)$, known as $\md p $ Wu classes, introduced in \cite{Wu}. We write $\mathtt{q}_k$ instead of $\mathtt{q}^p_k$ if the prime we consider is clear from the context. These classes are defined by the identity $\mathtt{q}_k(\xi)={\theta}^{-1}\cP^k\theta(1)$. Here, $\cP^n$  denotes the Steenrod's reduced $p$-th power operation and $\theta: H^{*}(X;\bbZ/p)\to  H^{*}(T\xi;\bbZ/p)$ denotes the Thom isomorphism. For more details on $\md p$ Wu classes we refer to \cite{milsta}, Ch.19. 

For each prime $p$, let  $\mathtt{q}^p_1(X)$ denote the negative of the $ \md p $ Wu class of $\nu_X$, the Spivak normal fibration of $X$ (which exists since $X$ is a finite complex, see \cite{spivak}). Given $S$ a set of primes, let $\KO(X)_{(S,\mathtt{q}_1)}$ denote the subset of $\KO(X)$ that consist of elements $\xi$ such that for each $p$ in $S$ the first $ \md p $ Wu class of $\xi$ satisfies the identity $\mathtt{q}^p_1(\xi)+\mathtt{q}^p_1(X)=0$ (or equivalently  $\mathtt{q}^p_1(\xi)=\mathtt{q}^p_1(\nu_X)$). Since the class $\mathtt{q}^p_1(X)$ is a homotopy type invariant of $X$ (see \cite{milsta} Ch. 19), the subset $\KO(X)_{(S,\mathtt{q}_1)}$ is invariant under the action of $\overline{\Aut_s}(\KO(X))$. We denote the quotient of this action by  $ \mathfrak{K}(X)_{(S,\mathtt{q}_1)} $. In particular, if $S=\emptyset$, then  $ \mathfrak{K}(X)_{(S,\mathtt{q}_1)}= \mathfrak{K}(X) $.

 \begin{notation}\label{notation}  $E^*_{*,*}({\xi})$ will denote the James spectral sequence for the stable homotopy as the generalized homology theory, identity fibration $id:X\to X$ and the stable bundle ${\xi:X\to BSO}$.  The abbreviations {\sf JSS} and {\sf AHSS} will be used for the James and Atiyah-Hirzebruch Spectral sequences respectively. For any finite spectrum $E$, $E^{\wedge}_{q}$ will denote the $q$-nilpotent completion of $E$ at the prime $q$ (also called localization at $\bbZ/q$, corresponding to localization at the Moore spectrum $M\bbZ_q$ of $\bbZ/q$), see \cite{bousfield}. 
\end{notation} 	
	
	\section{Main results}\label{prf}
Let $X$ be a simple Poincar\'{e} complex with formal dimension $m$. We impose the following condition on a  stable vector bundle $\xi:X\to BSO$:
\begin{condition}\label{cond1}
	For each $r\leq m$ the differential $d^{r}:E^r_{m,0}({\xi})\to E^r_{m-r,r-1}({\xi})$ in the James spectral sequence is zero.
\end{condition}
Observe that the image of the edge homomorphism of $E^*_{*,*}({\xi})$ in $ H_m(X,\bbZ)$ is the intersection of the kernels of all of the differentials with source $E^*_{m,0}({\xi})$, i.e. $im(\mathfrak{ed})=\bigcap_{r}\ker(d^{r})$. Thus, Condition \ref{cond1} implies that the group $E^2_{m,0}({\xi})=H_m(X;\bbZ)$ is equal to $E^\infty_{m,0}({\xi})$, i.e. edge homomorphism is surjective. For a given class $[\rho:M\to X]$ in $\Omega_m(\xi)$ we have $\mathfrak{ed}[\rho:M\to X]= \rho_*[M]$. Therefore, we can find a class $[\rho:M\to X]$ in $\Omega_m(\xi)$ such that $\mathfrak{ed}[\rho:M\to X]= \rho_*[M]$ is a generator of $H_m(X;\bbZ)$ with the preferred orientation. As a result, we get a degree one normal map $\rho:M\to X$, i.e. we have a surgery problem.

If Condition \ref{cond1} does not hold for $\xi$, i.e. we have a non-trivial differential $d^{r}:E^r_{m,0}({\xi})\to E^r_{m-r,r-1}({\xi})$ for some $r$, then the edge homomorphism can not be surjective. This means $\rho_*[M]$ can not be a generator of $H_m(X;\bbZ)$, i.e. $\rho$ can not be a degree one map. Hence, there is not a degree one normal map that represents a class in $\Omega_m(\xi)$. As a result, there is not a manifold simple homotopy equivalent to $X$ whose stable normal bundle lifts to $X$ through $\xi$. Hence, we have the following lemma:
\begin{lemma}\label{condlemma}
A stable vector bundle $\xi$ admits a degree one normal map if and only if Condition \ref{cond1} holds for $\xi$.
\end{lemma}
For the {\sf JSS}  for $\xi$, $E^*_{*,*}({\xi})$, there is a corresponding (isomorphic) {\sf AHSS} for the Thom spectrum $M\xi$, i.e. the  {\sf AHSS} whose $E^2$-page is $H_*(M\xi,\pi_*^S(*))$ which converges to $\pi_*(M\xi)$, with the isomorphism given by the Thom isomorphism. For a given prime $q$, it is well known that the $q$-primary part of $ \pi_k^S$ is zero whenever $0<k<2q-3$ (see  \cite{todastable}). We use finiteness of $\pi_k^S$ \cite{serre}. On each $\md q$ torsion part, the first non-trivial differentials of the {\sf AHSS} are given by the duals of the stable primary cohomology operations. Due to Wu formulas, when we pass to the {\sf JSS} we need to know  the action of Steenrod algebra on the Thom class. For $p= 2$ the action of Steenrod squares on the Thom class $U \in H^*(M\xi;\bbZ/2)$ is determined by the Stiefel-Whitney classes. In fact the $(\md 2)$ Wu formula asserts that $Sq^i(U)=U\cup w_i$ (see \cite{milsta}, p.91).

Let $S/q$ denote the homology theory given by $\bbS^{\wedge}_{q}$. Let $E$ be a spectrum. Consider the  {\sf AHSS} for the homology theory $S/q$, i.e. the coefficient groups  will be  $\pi_*(\bbS^{\wedge}_
{q} )$. Due to naturality of the {\sf AHSS}, the first non-zero differentials have to be stable primary cohomology operations independent of the generalized cohomology theory, see pp. 208 \cite{ahss}. For each $i$ with $0<i< 2q-3$ we have  $\pi_{i}(\bbS^{\wedge}_{q})=0$ and $\pi_{2q-3}(\bbS^{\wedge}_{q})=\bbZ/q$. Thus, the first non-trivial differential in this {\sf AHSS} appears at the $(2q-2)$-th page. This differential has to be a stable primary cohomology operation.  The only $\md q$ operations in this range are $0$ and dual of the $\md q$ Steenrod operation $\cP^1$. As in the proof of Lemma in  \cite[pp. 751]{teichner}, letting $E=\Sigma^{2q-2}H\bbZ/p$ as a test case, one can see that $d^{2q-2}$ is not always zero. The $d^2$ differential in $E^*_{*,*}({\xi})$  is given by the dual of the map $ x\mapsto Sq^2(x)+ w_2(\xi) \cup x$, see \cite{teichner} Proposition 1. Let us write $\mathtt{q}_1$ for $\mathtt{q}_1^q$, where $q$ is a fixed odd prime. We obtain a similar formula for the first non-zero differentials in $E^*_{*,*}({\xi})$ acting on $ \md q $ torsion part. 
\begin{lemma}\label{d2q-2}
 For each $n\geq 2q-2$ the differential on the $ \md q $ part $\ d^{2q-2}:E^{2q-2}_{n,0}({\xi})\to E^{2q-2}_{n-2q+2,2q-3}({\xi})$ is equal to the dual of the map $$\delta:H^{n-2q+2}(X;\bbZ/q)\to H^{n}(X;\bbZ/q)$$ defined as $ x\mapsto \cP^1(x)+ \mathtt{q}_1(\xi) \cup x$, composed with $\md q$ reduction.
	\end{lemma}
	\begin{proof}
 Consider the {\sf AHSS} for $M\xi$ and $S/q$. In this case the coefficient groups of the {\sf AHSS} will be  $\pi_*(\bbS^{\wedge}_
	{q})$ and it will converge to  $\pi_*(M\xi^{\wedge}_
	{q})$. From above remarks, the differential $d^{2q-2}$ in the {\sf AHSS} for $M\xi$ and $S/q$, is the dual of the $\md q$ Steenrod operation  $$\cP^1:H^{n-2q+2}(M\xi,\bbZ/q)\to H^n(M\xi,\bbZ/q).$$  By the Thom isomorphism theorem an element of $H^*(M\xi,\bbZ/q)$ is of the form $U\cup \sigma$ where $\sigma \in H^*(X;\bbZ/q)$ and $U$ is the Thom class. On the passage to the {\sf JSS}, Cartan's formula implies $$\cP^1(U\cup \sigma)=U\cup \cP^1(\sigma) + \cP^1(U)\cup \sigma=U\cup \cP^1(\sigma)+U\cup \mathtt{q}_1(\xi)\cup \sigma$$ hence in the James spectral sequence these differentials become duals of the maps $\ \sigma\mapsto \cP^1(\sigma)+\mathtt{q}_1(\xi)\cup \sigma$ composed with $\md q$ reduction. 
\end{proof}
 
We have the following lemma for the differential $d^m$ with source $E^m_{m,0}({\xi})$:
\begin{lemma}\label{dm}
Let $q$ be a prime and $m$ be an odd number. Let  $X$ be a $\bbZ/q$-homology sphere with a given $\bbZ/q$-homology isomorphism $f:S^m\to X$. Then for any stable vector bundle $\xi:X\to BSO$ that is in the kernel of $f^*:\KO(X)\to \KO(S^m)$, the image of the differential $d^m$ in  $E^m_{0,m-1}({\xi})$ has trivial $q$-torsion.
\end{lemma}
 
\begin{proof}
 Let  $\epsilon:S^m\to BSO$ be the stable vector bundle given by the composition $\xi\circ  f$. The map $f:S^m\to X$ induces a map of spectra $Mf:M\epsilon \to M\xi$. Since $f$ is a $\bbZ/q$-homology isomorphism, the induced map is also $\bbZ/q$-homology isomorphism, due to Thom isomorphism. Both $M\xi $ and $M\epsilon$ are connective spectra and of finite type. The space $X$ is $q$-good (see Definition I.5.1 in \cite{bousfieldkan}) due to 5.5 of \cite{BKbulletin}. Thus, the induced map $Mf^{\wedge}_{q}:M\epsilon^{\wedge}_{q} \to M\xi^{\wedge}_{q}$ is a homotopy equivalence (see \cite{bousfield} Proposition 2.5 and Theorem 3.1). We have $f^*(\xi)=0 $ in $\KO(S^m)$. Hence, $\epsilon$ is a trivial bundle. The Thom spectrum $M\epsilon $ is then homotopy equivalent to the wedge of spectra $\bbS \vee \Sigma^m \bbS$, as it is the suspension spectrum of $S^m\vee S^0$. Recall that  $E^*_{*,*}({\epsilon})$ is the  {\sf JSS} for the identity fibration $S^m\to S^m $ and the trivial stable vector bundle. Hence,  $E^*_{*,*}({\epsilon})$ collapses on the second page. As a result, $q$-torsion in $E^m_{0,m-1}({\epsilon})$ survives to the $E^\infty_{0,m-1}({\epsilon})$. The result follows by comparing the $q$-torsion in $E^m_{0,m-1}({\epsilon})$, via $Mf$, with the $q$-torsion in  $E^m_{0,m-1}({\xi})$.\end{proof}

The following lemma is a partial converse of Lemma \ref{dm}:
\begin{lemma}\label{notinker}
Assume $q=2$ and  $m=1(\md 8)$ in Lemma \ref{dm}. Then, if $\xi\notin \ker(f^*)$, then Condition \ref{cond1} does not hold for $E^*_{*,*}({\xi})$.
\end{lemma}

\begin{proof} Suppose that $\xi\notin \ker(f^*)$ and Condition \ref{cond1} holds for  $E^*_{*,*}({\xi})$. Let $\xi\circ f=\mu:S^m\to BSO$ be the nontrivial element in $\KO(S^m)$. Then, by a comparison as in the proof of Lemma \ref{dm}, Condition \ref{cond1} holds for $E^*_{*,*}({\mu})$. Thus, there is a degree one normal map $\rho:M\to S^m$ representing a class in $\Omega_m(\mu)$. Since $L^s_m(\bbZ)=0$, see \cite{kervmiln}, every such map can be surgered to a simple homotopy equivalence, $\tilde\rho:\tilde{M}\to S^m$. However, it is well known that every homotopy sphere is stably parallelizable (see \cite{kervmiln}). Hence, we get a contradiction, as $\mu\circ \tilde\rho$ is nontrivial in $\KO(S^m)$.
	\end{proof}

\begin{proof}[Proof of Theorem \ref{main}]
Since both $|\pi_1(X)|$ and $m$ are odd, due to Theorem 1 in \cite{bak}, the surgery obstruction groups vanish. Hence, every degree one normal map can be surgered to a homotopy equivalence. We will show that elements in $\KO(X)_{(S,\mathtt{q}_1)}$ are the ones that admit a degree one normal map.

Let $[\xi]$ be an orbit in $\mathfrak{K}(X)_{(S,\mathtt{q}_1)}$ represented by $\xi:X\to BSO$.  Consider the James spectral sequence, $E^*_{*,*}(\xi)$. Let $q$ be a prime with $ q < (m+2)/2 $ and $q\notin S$. Then $X$ is a $\bbZ/q$-homology sphere. By Lemma \ref{dm} the image of $d^m$ has trivial $q$-torsion. Since $E^r_{m-r,r-1}(\xi)=H_{m-r}(X;\pi_{r-1}^S)$ does not contain $q$-torsion when $r<m$, image of the differentials $d^r:E^r_{m,0}(\xi)\to E^r_{m-r,r-1}(\xi)$ have trivial $q$-torsion as well. Hence, all of the differentials based at $E^r_{m,0}(\xi)$ have trivial $q$-torsion in their images.

Now, let $p\in S$. Then $ 2p-2< m< 4p-4 $. It is well known that for $t<4p-5$ the $p$-torsion in $\pi_t^S$ vanishes except when $t=2p-3$ (see for example \cite{todastable}, III Theorem 3.13, B). Since $2p-2<m <4p-4 $, $E^m_{0,m-1}$ has trivial $p$-torsion, we have $d^m=0$. Hence, the only differential whose image can contain mod-$p$ torsion appears at degree $2p-2$. By Lemma \ref{d2q-2}, the differential  $d^{2p-2}$ on the $p$-torsion part is equal to the dual of the map $$\delta:H^{m-2p+2}(X;\bbZ/p)\to H^{m}(X;\bbZ/p)$$ defined as $ x\mapsto \cP^1(x)+ \mathtt{q}_1(\xi) \cup x$, composed with $(\md p)$ reduction. Let $x$ be an element in $H^{m-2p+2}(X;\bbZ/p)$. By Poincar\'{e} duality, there exists an element $s$ in $H^{2p-2}(X;\bbZ/p)$ such that $\cP^1(x)= s\cup x$ (see \cite{hirzebruch} Section 2). By definition $s=\mathtt{q}_1(X)$.  Then  $d^{2p-2}$ is trivial on mod-$p$ torsion as $\xi$ is an element in $\KO(X)_{(S,\mathtt{q}_1)}$, i.e. as $\mathtt{q}_1^p(X)+\mathtt{q}_1^p(\xi)=0$.
 
Now, assume that  $\mathtt{q}_1^p(X)+\mathtt{q}_1^p(\xi)\neq 0$. Then by  Poincar\'{e} duality there exist an element $a\in H^{m-2p+2}(X;\bbZ/p)$ such that $a\cup (\mathtt{q}_1^p(X)+\mathtt{q}_1^p(\xi))$ is nonzero in $H^{m}(X;\bbZ/p)$, i.e. $d^{2p-2}\neq 0$.

As a result, Condition \ref{cond1} holds for $\xi$ if and only if $\xi\in \KO(X)_{(S,\mathtt{q}_1)}\cap \ker(f^*)$. By Lemma \ref{condlemma}   $\xi\in \KO(X)_{(S,\mathtt{q}_1)}\cap \ker(f^*)$ if and only if $\xi$ admits a degree one normal map.  Hence, we can do surgery and obtain a smooth manifold $M$ with a simple homotopy equivalence $\omega: M\to X$ that represent a class in $\Omega_m(\xi)$ if and only if  $\xi\in \KO(X)_{(S,\mathtt{q}_1)}\cap \ker(f^*)$. By Lemma \ref{dm}, the image of $\Psi $ consists of orbits in $ \mathfrak{K}(X)_{(S,\mathtt{q}_1)} $ represented by elements in the kernel of $f^*:\KO(X)\to \KO(S^m)$.

In the case when $m\neq 1(\md 8)$ we have $\KO(S^m)=0$ by Bott periodicity, i.e.  image of $\Psi$ is $ \mathfrak{K}(X)_{(S,\mathtt{q}_1)} $. If $S=\emptyset$, then $ \mathfrak{K}(X)_{(S,\mathtt{q}_1)}=\mathfrak{K}(X)$, i.e.  $\Psi$ is a surjection. This completes the proof. \end{proof}

The following corollary is essentially stronger than the main result.
\begin{corollary}\label{freeproduct}
	Let $m$ and $S$ be as in Theorem \ref{main}. Suppose that $X$ and $X'$ are $\bbZ/q$-homology $m$-spheres fitting into a zigzag $\displaystyle S^m\stackrel{f}\to X'\stackrel{g}\leftarrow X$, so that both $f$ and $g$ are $\bbZ/q$-homology isomorphisms for every prime $q< (m+2)/2$ with $q\notin S$. If $\pi_1(X)$ is a free product of finitely many odd order groups, then the image of $\Psi$ contains all orbits in $\mathfrak{K}(X)_{(S,\mathtt{q}_1)} $ which are represented by elements in $ g^*(\ker(f^*))$, where $f^*$ and $g^*$ are the induced maps $\displaystyle \KO(S^m)\stackrel{f^*}\leftarrow \KO(X')\stackrel{g^*}\to \KO(X)$. 
\end{corollary}
\begin{proof}
	Assume $q\notin S$ with  $q< (m+2)/2$. For a bundle $\xi'$ over $X'$, Condition \ref{cond1} holds for $\xi'$ if and only if $\xi'\in \KO(X')_{(S,\mathtt{q}_1)}\cap \ker(f^*)$ for each $q$. One can compare spectral sequences for $\xi'$ and $g^*(\xi')$ as in the proof of Lemma \ref{dm}, and show that $d^m$ differential is trivial on $E^r_{m,0}(g^*(\xi'))$ whenever $\xi'\in \ker(f^*)$ for each such prime $q$. Thus, Condition \ref{cond1} holds for  $g^*(\xi')$, whenever $\xi'\in \KO(X')_{(S,\mathtt{q}_1)}\cap \ker(f^*)$. Repeating the arguments of Lemma \ref{notinker}, one can show that if $\xi'\notin \ker(f^*)$, then Condition \ref{cond1} does not hold for $g^*(\xi')$. The result follows from Theorem 5 in \cite{cappell}, together with Lemma \ref{condlemma} above.\end{proof}
Corollary \ref{freeproduct}, for example, allows us to do similar estimations for connected sums of manifolds satisfying the assumptions of Theorem \ref{main}. 

\begin{remark}
It can be seen from the proof  of Theorem \ref{main} (resp. Corollary \ref{freeproduct}) that we do not need a single map $f$ (or $g$) which is  simultaneously  a  $\bbZ/q$-homology isomorphism for every prime $q< (m+2)/2$ with $q\notin S$. It is enough that for every prime $q< (m+2)/2$ with $q\notin S$ there exist maps $f_q$ and $g_q$ (depending on $q$) which are $\bbZ/q$-homology isomorphisms. In this case, we need to replace $\ker(f^*)$ (or $g^*(\ker(f^*))$) by intersection over $q$ of all  $\ker(f_q^*)$ (or $g_q^*(\ker(f_q^*))$).  
\end{remark}

Let $\Theta_m$ denote the group of homotopy $m$-spheres. For any smooth $m$-manifold $M$, there is a subgroup $I(M)$ of $ \Theta_m$ called the inertia group of $M$, defined as $\{\Sigma\in \Theta_m: \Sigma \# M\cong M\}$, where $\cong$ here means diffeomorphic (see \cite{browder}). Two manifolds $M_1$ and $M_2$ are said to be almost diffeomorphic if there is a $\Sigma\in \Theta_m$ such that $M_1\# \Sigma^m\cong M_2$. It is known that almost diffeomorphic manifolds have isomorphic stable normal bundles, as homotopy spheres are stably parallelizable (see \cite{kervmiln}). Thus, their images are the same under $\Psi$. In order to determine the set of manifolds that are almost diffeomorphic to $M$, one needs to compute $I(M)$. Hence, to determine $\sM(X)$, it is necessary to know $I(M)$ for every $[M]$ in $\sM(X)$. It is known that $I(M)$ is not a homotopy type invariant, in fact it is not even a PL-homeomorphism type invariant, see for example \cite{DeSapio}. As a result, complete determination of $\sM(X)$ may not be possible in this generality.

The following corollary says that framed manifolds do not bound in ${\Omega_m(\xi)}$ for some $\xi:X\to BSO$.
\begin{corollary}\label{ed}
	Under the assumptions of Theorem \ref{main} together  with $m\neq 1(\md 8)$ and $S=\emptyset$, the  edge map  $\mathfrak{\bar{ed}}:\pi_*(\bbS)\to \pi_*( M\xi) $ is an inclusion for any stable vector bundle $\xi:X\to BSO$.
\end{corollary}
\begin{proof}
	As in the proof of Theorem \ref{main} for a prime $q$ the first differential in  $E^*_{*,*}({\xi})$ that acts non-trivially on $q$-torsion appears in  dimension $2q-2$. Since $S=\emptyset$, $X$ is a $\bbZ/q$-homology sphere for every prime $q$ with $2q-2<m$. Hence, $d^r=0$ for every $r<m$. As in the proof of Lemma \ref{dm}, by comparing with $E^*_{*,*}({\epsilon})$ we get $d^m=0$ (as  $E^*_{*,*}({\epsilon})$ collapses at the second page, due to degree reasons). Therefore, the first nontrivial differential  appears when $r\geq m+1$. But then the target of $d^r$ should be zero. Hence, $E^*_{*,*}({\xi})$ collapses at the second page, and we get that $\mathfrak{\bar{ed}}:\pi_*(\bbS)\to \pi_*( M\xi) $ is  an inclusion.
\end{proof}
Observe that the degree of $f$ (as in Theorem \ref{main}) plays the important role here, as it is co-prime to smaller primes. One can ask what the necessary and sufficient  conditions are on the pair $(X,\xi)$ so that the natural map $\mathfrak{\bar{ed}}:\pi_*^s\to \pi_{*}( M\xi )$ induced by the inclusion of point is injective. It is well known that such is not the case for classical Thom spectra like $MO$ or $MSO$ (see \cite{thom}). In the case when $\xi$ is a trivial bundle, there are examples for which this is true. Another possible question is: For which spaces $X$, this natural map $\pi_*^s\to \pi_{*}( M\xi )$  is injective for every  stable vector bundle $\xi:X\to BSO$. Corollary  \ref{ed} provides just one such example. 

Suppose that $q=2$ and  $m=1(\md 8)$ in Lemma \ref{dm}. In this case $f$ induces the identity on cohomology with coefficients $\bbZ/2$. Bott periodicity theorem asserts that $\KO(S^m)=\KO^{-m}(S^0)=\bbZ/2$. The map $f$ induces a map on the Atiyah-Hirzebruch spectral sequences. At the second page we have $f^*:H^m(X;\bbZ/2)\to H^m(S^m;\bbZ/2)$, which is an isomorphism.  The mod-$2$ class in $ H^m(S^m;\bbZ/2) $ survives to the infinity page of the Atiyah-Hirzebruch spectral sequence for $\KO(S^m)$. Hence,  $f^*$ is a surjection on the infinity page by the naturality of the {\sf AHSS}. It follows that $f^*:\KO(X)\to \KO(S^m)$ is surjective (for the case when $X$ is a spherical space form, this follows from \cite{karoubi}, Theorem 1-(b)). Hence, we have the following remark:
\begin{remark}\label{ker}
	In the case when $q=2$ and  $m=1(\md 8)$ in Lemma \ref{dm}, we have $[\KO(X):\ker(f^*)]=2$.
\end{remark}
Since $2\notin S$ in Theorem \ref{main}, we have $[\KO(X):\ker(f^*)]=2$ as well. Thus, we can determine the image in the case when $m=1(\md 8)$ as well.



\section{Examples}\label{examples} 
Let $L^k(n)$ denote the quotient space $S^{2k+1}/\bbZ/n$ of the free linear action of  $\bbZ/n$ on $S^{2k+1}$. If $(q,n)=1$, then $L^k(n)$ is a $\bbZ/q$-homology sphere with the covering projection being the $\bbZ/q$-homology isomorphism. Let $L^k(n,\mu)$ denote the orbit space of a free action $\mu$ of $\bbZ/n$ on $S^{2k+1}$ where $\mu$ acts by homeomorphisms. Such $L^k(n,\mu)$  are often called fake lens spaces (see \cite{atlas} and \cite{macko}, \cite{bpw} for more details on topological and \cite{petrie2}, \cite{orlik} for smooth fake lens spaces). For any such action $\mu$, one can always find a lens space $L^k(n)$ homotopy equivalent to $L^k(n,\mu)$ (see \cite{bpw} P.456). If $p$ is a prime and $k$ is an integer with $k\leq 2p-3$, Theorem \ref{main} applies to any fake lens space $L^k(p,\mu)$. If $k< p-1$, then $S=\emptyset$ and if  $p-1\leq k\leq 2p-3$, then $S=\{p\}$ (note that dimension of $L^k(p,\mu)$ is $2k+1$). In this case, if $T=\{q\ \text{prime}: q\leq k+1\}\backslash \{p\}$, then for any prime $q\in T$, $L^k(p,\mu)$ is a $\bbZ/q$ homology sphere and $K(L^k(p,\mu))$ does not have any element of order $q$ for $q$ odd. In general, if $n$ is a natural number not divisible by primes less than or equal to $\frac{k+2}{2}$, $\mu$ is an action of $\bbZ/n$ on $S^{2k+1}$, $S=\{p\ \text{prime}: p|n\}$ and $T=\{q\ \text{prime}: q\leq k+1\}\backslash S$, then Theorem \ref{main} applies to $L^k(n,\mu)$ where the image of $\Psi $ consists of orbits in $\mathfrak{K}(L^k(n,\mu))_{(S,\mathtt{q}_1)}$ that are represented by elements in $\ker(f^*)$, where $f$ is the covering projection. Again, for $q\in T$ odd prime, the group $K(L^k(n,\mu))$ does not have any $q$-torsion. We refer to Theorem 2 in \cite{kambe} for the $K$-theory of a lens space and to \cite{smallen} for calculation of the group of homotopy classes of self homotopy equivalences of a lens space. For the particular cases when $k=p-3$ and $ k= 2p-4$, we can get from Theorem 2.A in \cite{smallen} that $\Aut_s(\KO(L^k(p,\mu)))$ has only $2$-elements, namely the identity and the automorphism mapping an element to its algebraic inverse. Hence, in these cases each orbit (except the orbit of $0$) in $\mathfrak{K}(L^k(p,\mu)) $ has exactly two elements.

Another class of examples can be obtained from spherical space forms. There is a vast literature on classification of spherical space forms, see for example \cite{mtw1}, \cite{mtw2} and \cite{mtw3}. Let $\Sigma$ be a homotopy $m$-sphere with $m\geq 5$. Let $\pi$ be a group that can act freely and smoothly on $\Sigma $ and let $X=\Sigma/\pi$, so that $f:\Sigma\to X$ is a principal $\pi$-bundle. Let $p\geq 3$ be the smallest prime dividing the order of $\pi$. Then there is a map $\varphi:X\to B\pi$ that classifies $f$. The group of self equivalences $ \Aut(X)$  of $X$ contains a normal subgroup isomorphic to all inner automorphism $\Inn(\pi)$ of $\pi$ (see  \cite{smallen} Corollary 1 and Theorem 1.4). Note that, an inner automorphism induces the identity on all (generalized) cohomology groups of $B\pi$, due to the commutativity in cohomology. Let  $\alpha:X\to X$ in $\Aut(X)$ be a self equivalence. Consider the diagram
\[\begin{tikzpicture}[scale=.8]
\node (0) at (1,0) {$X$};
\node (c) at (3,0) {$B\pi$};
\node (a) at (1,2) {$X$};
\node (b) at (3,2) {$B\pi$};
\path[->,font=\scriptsize,>=angle 90]
(a) edge node[above]  {$\varphi$} (b)
(b) edge node[right] {$\alpha_*$} (c)
(a) edge node[left] {$\alpha$} (0)
(0) edge node[above] {$\varphi$} (c);
\end{tikzpicture}\]
so that $ \alpha$ and  $ \alpha_*$ induce the same map on $\pi_1(X)=\pi$. By an argument as in \cite{davislect} Theorem 7.26, the diagram commutes up to homotopy. It is well-known that   $\varphi$ induces a surjection on $\KO$ (see for example \cite{karoubi}). Thereby, inner automorphisms of $\pi$ induce identity on $\KO(X)$ as well. Denote by $\Out(X)=\Aut(X)/\Inn(\pi)$ \emph{the group of outer self-equivalences of $X$}. Then, in order to determine  $\overline{\Aut_s}(\KO(X))$, we only need to consider automorphisms induced by self equivalences belonging to a fixed set of representatives of cosets in $\Out(X)$.

For $m\neq 1 (\md 8)$ and $ m<2p-2$, then part 1 of Theorem \ref{main} applies to $X$, so that $\Psi:\sM(X)\to \mathfrak{K}(X) $ is surjective. In this case $S=\emptyset$. For $m<4p-4$ and no other prime between $p$ and $2p$ divides the order of $\pi$, then the image of  $\Psi$  is determined by the first $\md\text{-} p$ Wu class of the Spivak normal bundle of $X$. In general, if $S$ is the set of primes  between $p$ and $2p$ which divide the order of $\pi$, then the image of $\Psi$ is  $ \mathfrak{K}(X)_{(S,\mathtt{q}_1)}$.
We refer to \cite{karoubi} for the computation of $\KO(X)$ (for $X=\Sigma/\pi$ as above) and the results given in \cite{golasinski1} (and although indirectly, in \cite{golasinski2}) for the computation of $\Aut(X)$. Of course we only need these computations when $\pi_1(X)$ is of odd order. The action of $\overline{\Aut_s}(\KO(X))$ on $\KO(X)$ is given by the restriction of the usual canonical action of the automorphism group $\Aut(\KO(X))$ on $\KO(X)$, which can be understood once $\KO(X)$ is known.

Let $m$ and $S$ be as in Theorem \ref{main}. Let $X_0$ and $X_1$ with given maps $f_i:S^m\to X_i$ for $i=0,1$ be two $\bbZ/q$-homology $m$-spheres, so that both $f_0$ and $f_1$ are $\bbZ/q$-homology isomorphisms for $q<(m+2)/2$ with $q\notin S$, i.e. Theorem \ref{main} applies to  both $X_0$ and $X_1$. Then, $X_0\# X_1$ is also a $\bbZ/q$-homology sphere for primes $q<(m+2)/2$ with $q\notin S$ (which follows easily from Mayer-Vietoris sequence) and the fundamental group of  $X_0\# X_1$ is the free product $\pi_1(X_0)*\pi_1(X_1)$ (which follows from a simple application of Van Kampen's Theorem). Thus, Corollary \ref{freeproduct} applies to the connected sum $X_0\# X_1$. If there exists a map $f:S^m\to X_0\# X_1$ satisfying the conditions of Theorem \ref{main}, then we can choose $g$ as the identity map. In this case,  $\Image(\Psi)$ consists of all orbits in $\mathfrak{K}(X_0\# X_1)_{(S,\mathtt{q}_1)} $ which are represented by elements in $\ker(f^*)$. If we can not find such a map $f$, then we can apply Corollary \ref{freeproduct} by using the zigzags $\displaystyle S^m\stackrel{f_i}\to X_i\stackrel{g_i}\leftarrow X_0\# X_1$, where $g_i$'s are the obvious collapse maps. In this case, $\Image(\Psi)$ contains all orbits in $\mathfrak{K}(X_0\# X_1)_{(S,\mathtt{q}_1)} $ which are represented by elements in $ g^*(\ker(f^*))$.

For a given finite $CW$-complex $X$, denote by $\overline{\Aut}(K^i(X))$ the subgroup of ${\Aut}(K^i(X))$ that consists of automorphisms induced by elements in $\Aut(X)$ (we simply write $\overline{\Aut}(K(X))$ when $i=0$). Due to the naturality of suspension isomorphism and Bott periodicity, we can identify $\overline{\Aut}(K(X)) $ with $\overline{\Aut}(K(\Sigma^{8i} X))$.
Hence, the natural map from $\Aut(X)$ to $\overline{\Aut}(K(X))$ factors through the group of stable self equivalences of $X$, which is equal to $\operatornamewithlimits{colim}_i{\Aut}(\Sigma^{8i} X)$ (see for example \cite{kahn}, \cite{petar} and \cite{johnston} for more details about the group of stable self equivalences). If $X$ is a $\mathbb Z/q$-homology sphere for a prime $q$, then all Betti numbers of $X$ are less than or equal to $1$. Thus, due to Theorem 1.1-(a) in \cite{johnston}, the group of stable self equivalences of $X$ is finite, which implies that $\overline{\Aut_s}(K(X))$ (which is a subgroup of $\overline{\Aut}(K(X))$) is finite.


\subsection*{Acknowledgment}
{I wish to thank \"Ozg\"un \"Unl\"u and Matthew Gelvin for their valuable advices. I also thank the anonymous referee for valuable comments. This research was partially supported by T\"UB\.ITAK-B\.IDEB-2214/A Programme.}

\bibliographystyle{plain}
\bibliography{mybibfile}

\end{document}